\documentclass[preprint,12pt]{elsarticle}
\usepackage{amsthm}
\usepackage{enumitem}
\newtheorem{Definition}{Definition}
\usepackage{amsmath, amssymb}  
\newcommand{\R}{\mathbb{R}}
\newtheorem{remark}{Remark}[section]  
\newtheorem{example}{Example}[section]  
\usepackage{float}  
\newtheorem{proposition}{Proposition}[section]

\usepackage{amssymb}
\usepackage{amsmath}

\journal{}

\begin{document}

\begin{frontmatter}



\title{Exact Solutions for Classes of Nonlinear Differential Equations on Fractal Supports.} 


\author[1]{Donatella Bongiorno\corref{mycorrespondingauthor}}
\cortext[mycorrespondingauthor]{Corresponding author}
\ead{donatella.bongiorno@unipa.it}

\author[2]{Alireza Khalili Golmankhaneh}
\ead{alireza.khalili@iau.ac.ir}

\address[1]{Dipartimento di Ingegneria, Università di Palermo,
Viale delle Scienze, Ed. 8, 90128 Palermo, Italy}

\address[2]{Department of Physics, Ur. C., Islamic Azad University,
Urmia 63896, West Azerbaijan, Iran}

\begin{abstract}
In this paper, the exact solutions of certain non-linear differential equations defined on a fractal subset of the real line are presented. Particular attention is paid to the Riccati-type fractal differential equation, for which a connection with the Schrödinger equation is also provided.
\end{abstract}




\end{frontmatter}



\section{Introduction}

The study of non-linear differential equations (NDE) is a cornerstone of mathematical physics and engineering, providing essential models for complex phenomena across numerous disciplines, from fluid dynamics and quantum mechanics to biology and finance \cite{Coddington, Verulst, edwards2000differential}.
 It is known that powerful analytical and numerical techniques exist for NDEs defined on standard Euclidean domains. Among the non-linear ordinary differential equations, the Riccati differential equation is a notable example. This equation, with quadratic right-hand sides, is closely linked to the calculus of variations and optimal control theory, playing a key role in the optimal control of complex networks \cite{bellon2008riccati,Martin,Pavon}.
Despite these advances, a significant challenge arises when the underlying domain possesses a non-integer, or fractal, dimension.

This challenge is central given the relevance of fractal geometry in modeling and understanding natural phenomena—such as blood vessels, coastlines, mountains, and clouds—as highlighted in the classic exposition by B. Mandelbrot \cite{Mandelbro}. These fractal structures exhibit unique features, including self-similarity and fractal dimensions greater than their topological dimensions, which distinguish them from traditional Euclidean objects. Consequently, conventional metrics (e.g., length, surface area, and volume) typically applied to Euclidean forms have proven insufficient for analyzing the properties of analytic functions defined on a fractal set or on a fractal curve \cite{Qaswet,falconer1999techniques,feder2013fractals}.

To address this, mathematical methods extending beyond classical analysis, such as harmonic analysis \cite{kigami1989harmonic, kigami2001analysis, jorgensen2006analysis}, measure theory \cite{jiang1998some, bongiorno2015integral}, stochastic processes \cite{Barlowqq1} and fractional calculus \cite{valarmathi2023variable}, were developed. Specifically, a method known as fractal calculus ($F^\alpha$-calculus) was formulated for fractal subsets of the real line \cite{parvate2009calculus} and fractal curves \cite{parvate2011calculus}, offering a framework highly similar to the classical one.

In this paper, we adopt the $F^\alpha$-calculus framework, initially introduced in \cite{parvate2009calculus} and later refined by A.K.G. in \cite{Alireza-book}. The parameter $\alpha$, denoting the fractal dimension, is proven to coincide with the well-known Hausdorff dimension when the fractal set $F$ is compact. Central to this formulation is the staircase function, which generalizes the Cantor staircase function and is the key to defining the fractal integral and derivative.

The application of such calculus has led to significant recent progress, including the development of new methods for solving fractal differential equations (FDEs) \cite{golmankhaneh2023solving, khalili2024fractal} and systems of FDEs \cite{golmankhaneh2025homogeneous,khalili2025fractal77}. These models have proven effective in simulating processes with memory, modeling of fractal integral equations via Volterra fractal operator and describing power-law behavior in complex systems such as sub- and super-diffusion \cite{khalili2024fractaldd, Uc, golmankhaneh2018sub}.

In this paper, we show how the $F^\alpha$-calculus offers new avenues for finding exact solutions to some type of non-linear FDEs on fractal supports.

Precisely, the Section \ref{S-2} prepares the reader for the subsequent material by reiterating and adapting several crucial definitions derived from the paper of A. Parvate and A. D. Gangal \cite{parvate2009calculus}, tailoring them to our specific computational environment. The accompanying remarks underline the mechanism by which fractal calculus extends the concepts of ordinary calculus. A definition of fractal-primitive and its characterization is also provided. In Section \ref{S-3} we focus on the resolution of non-linear fractal differential equations of the form $D^\alpha_F y(t)=g(\phi(y,t))$, where $g$ is an $F$-continuous function on a fractal set $F$ and $\phi(y,t)$ is a given function. In Section \ref{S-4} we demonstrate how $F^\alpha$-calculus can be applied to derive exact solutions for Riccati-type fractal differential equations (RFDE). This is highly significant because the design of regulators for complex networks, which are often modeled using fractal structures, typically relies on the Riccati equation. By demonstrating how to derive exact solutions for RFDEs, this paper opens a new research area for designing optimal regulators that explicitly account for the underlying fractal topology of the network in resource allocation and stability analysis.
Furthermore, in Section \ref{S-5}, we show the fractal Riccati formulation of the Schrödinger equation, highlighting through significant examples how the RFDE plays an important role in quantum mechanics as well. Our model opens a new research line in a topic that in recent years has prompted a widespread and still growing interest: the Nonlinear Schr\"{o}dinger Equation (NLS) on non-standard domains. It will thus be possible to study models where dispersion and nonlinearity effects in fractal domains are balanced, such as describing the propagation of light pulses and optical solitons in optical fibers, Bose-Einstein Condensates, and the physics of water waves. Finally, Section \ref{S-6} concludes the paper by suggesting directions for future research.

\section{Preliminary\label{S-2}}
In this section, we provide a brief overview of $F^{\alpha}$-calculus, based on some definitions given in \cite{ parvate2009calculus} and in \cite{Alireza-book}, here suitably restated. 
Throughout all the paper we denote by \( F\)  a compact\ \(\alpha\)-perfect fractal subset of the real line, where
 \( \alpha \in (0,1] \) is its  fractal dimension. Moreover denoted by $[a,b]$ an interval of the real line, we assume that $F \cap [a,b]\neq \emptyset$.
 
\begin{Definition}\label{ex} 

Let $0<\alpha \leq 1$. 
The $\alpha$-\textit{dimensional Hausdorff measure} of a subset $A$ of the real line is defined as:
$$\mathcal{H}^{\alpha}(A)=\lim_{\delta\rightarrow 0}\inf\left\lbrace{\sum_{i=1}^{\infty}(\mbox{diam}(A_{i}))^{\alpha}:A\subset \bigcup_{i=1}^{\infty}A_{i},\, \mbox{diam}(A_{i})\leq\delta}\right\rbrace.$$

Moreover, the unique number $\alpha$ for which $\mathcal{H}^{t}(A)=0\,$ if $\, t>\alpha\,$ and $\,\mathcal{H}^{t}(A)=\infty\,$ if $\, t<\alpha\,$ is called the fractal (Hausdorff) dimension of $A$.
\end{Definition}

\begin{Definition}
Let $[a,b]$ be an interval of the real line and let $t\in F \cap [a,b].$
The Staircase function associated with the fractal set $F$ of order $\alpha$ is defined by 

$$S^{\alpha}_F(t)\,=\left \{ 
 \begin{array}{l}
\Gamma(\alpha+1)\mathcal{H}^{\alpha}(F\cap [p,t]), \ \ \ \ \ \ \ \ \ \ \ \textrm{if}\ \ t \geq p \\
\\
- \Gamma(\alpha+1)\mathcal{H}^{\alpha}(F\cap [t,p])\ \ \ \ \ \ \ \ \ \  \textrm{if}\ \  t<p.
\end{array}
\right.$$

\medskip

 where $p\in [a,b]$ is fixed and $\Gamma (\cdot)$ is the well known gamma function.  
  \end{Definition}

\begin{Definition}
Let $[a,b]$ be an interval of the real line. Let $t\in F \cap [a,b]$ and let $r\in \R^+.$
A fractal neighborhood with center $t$ and radius $r$ is the fractal interval of the form $V_F(t,r)= (t-r,t+r)\cap (F\cap [a,b]).$
\end{Definition}
 
\begin{Definition}
Let $f:[a,b]\to\R$.  We say that the function $f$ is $F$-continuous at a point $t\in F\cap [a,b],$
if for every fractal neighborhood $V_F$ of $f(t)$ there exists a fractal neighborhood $W_F$ of $t$ such that $f(t)\in V_F$ whenever $t\in W_F.$  In other words $f$ is $F$-continuous at a point $t\in F$ if 
the following fractal limit  exists  $$\underset{ y\rightarrow t}{F_{-}\text{lim}}~f(y)\,=\,f(t).$$ Whenever $f$ is $F$-continuous at every point of $F\cap [a,b],$ then $f$ is called a $F$-continuous function. The set of all such functions is denoted by $ \mathcal{C}(F\cap [a,b]).$

\end{Definition}

\begin{remark}
 Note that the previous definition does not involve values of the function $f$ at a point $y$ if $y\notin F\cap [a,b].$
 Therefore, the notion of $F$-continuity is a generalization of the classical notion of continuity. 
\end{remark}

\begin{Definition}
Let $f:[a,b]\to\R$.
The fractal derivative of a function $f$ at a point $t\in [a,b]$ is defined as:
\begin{equation}
D_F^\alpha f(t) = \left \{\begin{array}{ll}
 \underset{ y\rightarrow t}{F_{-}\text{lim}}~ \frac{f(y) - f(t)}{S_F^\alpha(y) - S_F^\alpha(t)},    & t\in F\cap [a,b] \\
   0,  & otherwise
\end{array}\right.
\end{equation}

\bigskip
where \( S_F^\alpha(t) \) is the Staircase function of order \(\alpha\) defined for the set \( F \).  

Whenever $f$ has a fractal derivative at every point $t\in [a,b]$ we say that $f$ is $F^{\alpha}$-derivable on $[a,b]$.
\end{Definition}

\begin{proposition} \label{pp1}
Let $f:[a,b]\to \R$ and $g:[a,b]\to\R$ be two real functions that are  $F^{\alpha}$-derivable on each point $t\in [a,b]$. Then, we have:
$$D^{\alpha}_F(f+g)(t)= D_F^\alpha f(t)\,+\,D_F^\alpha g(t), $$
     $$D_F^\alpha (fg)(t)=f(t)D_F^\alpha g(t)+g(t)D_F^\alpha f(t), $$
$$D_F^\alpha \left (\frac{f}{g}\right)(t)=\frac{g(t)D_F^\alpha f(t)-f(t)D_F^\alpha g(t)}{g^2(t)}. $$
  
\end{proposition}

\begin{Definition}\label{999}

The characteristic function  $\chi_{F}:[a,b]\to\R$ of the fractal set $F$ is defined by
\begin{equation}\label{w}
  \chi_{F}(t)=\left\{
                \begin{array}{ll}
                  1, & t\in F\cap [a,b]; \\
                  0, & otherwise.
                \end{array}
              \right.
\end{equation}
\end{Definition}

\begin{remark}\label{*} 
Let us observe that: 

\begin{itemize}
    \item [$\star$] $\chi_{F}\in \mathcal{C}(F\cap [a,b]),$
    \item[$\star$] $D^{\alpha}_FS^{\alpha}_F(t)= \chi_{F}(t), \quad \forall t \in [a,b]$
\end{itemize}

\end{remark}

\begin{Definition}

Let $f:[a,b]\to\R$. If $f$ is $F^{\alpha}$-derivable at every point $t\in [a,b]$, then the $F^{\alpha}$-derivative function $D_F^\alpha f:[a,b]\to\R$ is well defined.
\begin{enumerate}

    \item If  $D_F^\alpha f(t)$ is $F$-continuous therefore we say that $f$ belongs to the space $C^1(F\cap [a,b])$. 

    \item If $D_F^\alpha f(t)$ is $F^{\alpha}$-derivable at $t\in F,$ we say that $f$ has a $2\alpha$-fractal derivative at $t,$ denoted by $D_F^{2\alpha} f(t):=D_F^\alpha (D_F^\alpha f (t))$.
\end{enumerate}
    
\end{Definition}

\begin{remark}\label{**} 
It is trivial to observe that:
$$D^{2\alpha}_FS^{\alpha}_F(t)= D_F^\alpha (D_F^\alpha S^{\alpha}_F(t))=D^{\alpha}_F \chi_{F}(t)=0, \quad \forall t\in [a,b]$$
\end{remark}

\begin{Definition}
Let $f:[a,b]\to\R$ and let $\Psi:[a,b]\to\R$ be two real functions. We say that $\Psi$ is a fractal-primitive of $f,$ if 
 $\Psi$ is $F^{\alpha}$-derivable on $[a,b]$  and we have that 
\begin{equation}\label{w}
  D^\alpha_F(\Psi(t))=\left\{
                \begin{array}{ll}
                  f(t), & t\in F\cap [a,b]; \\
                  0, & otherwise.
                \end{array}
              \right.
\end{equation}
 
 \end{Definition}

\begin{example}
    By Remark \ref{*} it follows that the Staircase function associated with the fractal set $F$ of order $\alpha$ is a fractal-primitive of the characteristic function  $\chi_{F}$.
\end{example}

 \begin{proposition}
 Let $f:[a,b]\to\R$.
If $f(t)$ admits a fractal-primitive $\Psi(t)$ on every point $t\in F\cap [a,b]$, then for every constant $S^{\alpha}_F (c)\in \R$, the function $\Phi(t)=\Psi(t)+ S^{\alpha}_F (c)$ is a fractal-primitive of $f(t)$. 
\end{proposition}

The proof is straightforward.

\begin{Definition}
Let $f:[a,b]\to \R$ and let $t\in [a,b].$ The set of change of $f,$ symbolized as $Sch (f)$, is the collection of all such points $t$ where the behavior of the function is locally non-constant. Formally: 
$$Sch(f)=\{t\in [a,b]:\forall\delta>0, \, \exists y_1,y_2\in (t-\delta, t+\delta)\cap [a,b], \, \textit{such that}\,\, f(y_1)\neq f(y_2)\}$$

\end{Definition}

\begin{example}
Let $S^{\alpha}_F (c)\in \R$ and let $f_1(t)= S^{\alpha}_F (c),$ therefore    $\text{Sch}(f_1) = \emptyset$.

Let  $f_2(t) = t,$ for every $t\in [a,b],$ therefore $\text{Sch}(f_2) = [a,b]$.
\end{example}

\begin{proposition}
Let $\Psi:[a,b]\to\R$  and $\Phi:[a,b]\to\R$ be two fractal-primitives of $f:[a,b]\to \R$. If $Sch(\Psi -\Phi)\subset F\cap [a,b],$ therefore there exists a constant $S^{\alpha}_F (c)\in \R$ such that $\Psi(t)=\Phi(t)+S^\alpha_F(c)$ for every $t\in F\cap [a,b].$
\end{proposition}

\begin{proof}
    Define $H(t)=\Psi (t)-\Phi (t),$ for all $t\in F\cap [a,b].$ By  hypothesis we have 
    $D^\alpha_F(H(t))=D^\alpha_F \Psi (t)-D^\alpha_F \Phi(t)=f(t)-f(t)=0,$ for all $t\in F\cap [a,b].$ The conclusion then follows by applying the Corollary 52 in \cite{parvate2009calculus} to the function $H$.
\end{proof}

\begin{Definition}
Let $f:[a,b]\to\R$.
The set of all fractal primitives of $f$ on $F\cap [a,b]$ is called the indefinite $F^{\alpha}-$integral and is denoted by the symbol 
$$ \int f(t) d_F^\alpha t  $$
\end{Definition}

\begin{example}
By  Remark \ref{*} it follows
 $$ \int \chi_F(t) d_F^\alpha t = S^{\alpha}_F(t) + S^{\alpha}_F(c)$$
\end{example}

\section{Solving  equations  of the form: $D^{\alpha}_F\ y(t)=g(\phi(y,t))$ \label{S-3}}

In this section, we address the class of homogeneous nonlinear fractal differential equations defined by:
\[
D^{\alpha}_F y(t) = g(\phi(y, t)), \quad \text{with } t \in F\cap [a, b]
\]
where \( \phi(y, t) \) is an assigned function. We establish that a carefully selected change of variables effectively reduces this complex nonlinear problem into a more tractable linear fractal differential equation with separable variables, thereby facilitating its study and analytical solution.

\subsection{First case: $\phi(y,t)=\dfrac{y(t)}{S^{\alpha}_F(t)}$}

The fractal differential equation we aim to solve has the form:
\begin{equation}
D^{\alpha}_F y(t) = g\left(\frac{y(t)}{S^{\alpha}_F(t)}\right), 
\end{equation}
To solve this equation, we make the substitution
\begin{equation}\label{e433}
y(t) = S^{\alpha}_F(t) z(t).
\end{equation}
Applying Proposition \ref{pp1} to both sides of Eq.~\eqref{e433}, we have:
\[
D^{\alpha}_F y(t) = z(t)\chi_F (t) + S^{\alpha}_F(t) D^{\alpha}_F z(t).
\]
Substituting this into the original equation, we obtain the following fractal differential equation:
\[
S^{\alpha}_F(t)\, D^{\alpha}_F z(t) + \chi_{F}(t) z(t) = g(z(t)),
\]
which can be solved by the method of separation of variables described in \cite{golmankhaneh2023solving}.

\begin{example}
Consider the fractal differential equation on the ternary Cantor set \( C \subset [0,1]\) given by:
\begin{equation}\label{iu657}
D^{\alpha}_C y(t) = 1 + \frac{y(t)}{S^{\alpha}_C(t)}.
\end{equation}
It is straightforward to observe that
\[
g\left(\frac{y(t)}{S^{\alpha}_C(t)}\right) \equiv 1+ \frac{y(t)}{S^{\alpha}_C(t)}.
\]
By setting \( z(t) = \dfrac{y(t)}{S^{\alpha}_C(t)} \) and applying Proposition \ref{pp1} to $y(t)=z(t)S^{\alpha}_C(t)$, we get:
\begin{equation}
D^{\alpha}_C y(t) = z(t)D^{\alpha}_C S^{\alpha}_C(t)+ S^{\alpha}_C(t) D^{\alpha}_C z(t),
\end{equation}
and thus, by Remark \ref{*}, we have:
\begin{equation}
z(t)\chi_C(t) + S^{\alpha}_C(t) D^{\alpha}_C z(t) = \chi_{C}(t) + z(t).
\end{equation}
Therefore, by Def.\ref{999} and by subtracting \( z(t) \) from both sides, we obtain:
\begin{equation}\label{333}
S^{\alpha}_C(t) D^{\alpha}_C z(t) = \chi_{C}(t).
\end{equation}
By the method of separation of variables, Eq.~\eqref{333} becomes:
\[
D^{\alpha}_C z(t) = \frac{1}{S^{\alpha}_C(t)}  \chi_C(t).
\]
Fractal integrating both sides yields:
\begin{equation}
z(t)  = \int \frac{\chi_C(t)}{S^{\alpha}_C(t)}\, d^{\alpha}_C t = \ln |S^{\alpha}_C(t)| + S_C^{\alpha}(c),
\end{equation}
where \( S_C^{\alpha}(c) \) is an arbitrary constant. Finally, since \( y(t) = S^{\alpha}_C(t) z(t) \), we obtain the exact solution:
\begin{equation}\label{23222}
y(t) = S^{\alpha}_C(t) \ln \left( S_{C}^{\alpha}(c') S^{\alpha}_C(t) \right),
\end{equation}
where \( S_{C}^{\alpha}(c') = \pm \exp(S_{C}^{\alpha}(c)) \). Note that in Figure~\ref{fig:SF_logSF}, we depict the graphical representation of the solution to Eq.~\eqref{iu657}.
\begin{figure}[H]
    \centering
    \includegraphics[width=0.8\textwidth]{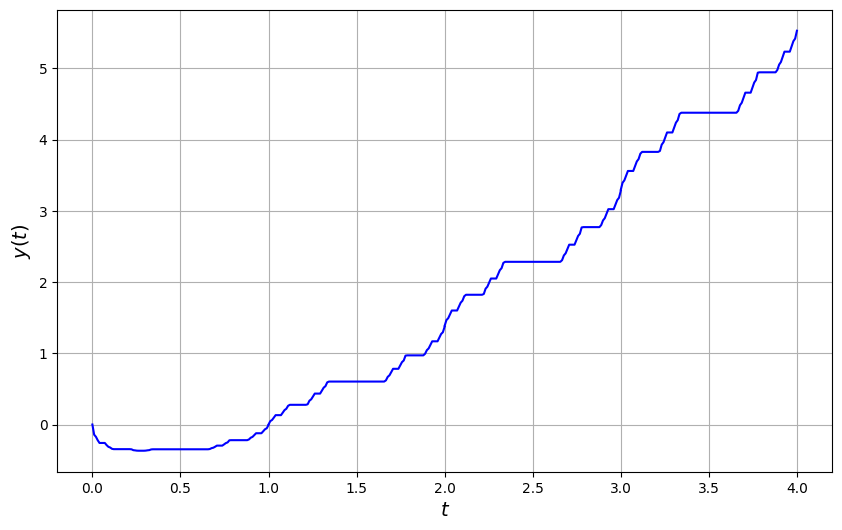}
    \caption{Plot of Eq.~\eqref{23222} over the ternary Cantor set with fractal dimension \( \alpha = \frac{\log 2}{\log 3} \). }
    \label{fig:SF_logSF}
\end{figure}
\end{example}

\subsection{Second case: $\phi(y,t)=a S^{\alpha}_F(t)+by(t)$ with $a$ and $b$ two non-zero real numbers.}

Let us consider a fractal differential equation of the type
\[
D^{\alpha}_F y(t) = g(a S^{\alpha}_F(t) + b y(t)), 
\]
where \( g \) is an \( F \)-continuous function  and \( a \), \( b \) are two non-zero real constants. This equation can be reduced to a fractal differential equation with separable variables by the following substitution:
\[
z = a S^{\alpha}_F(t) + b y(t).
\]

In fact, since \( D^{\alpha}_F z(t) = a \chi_F(t) + b D^{\alpha}_F y(t) \), we obtain the equivalent fractal differential equation in the unknown variable \( z \):
\[
D^{\alpha}_F z(t) = a \chi_F(t) + b g(z(t)),
\]
which can be solved easily using the separation of variables method (see \cite{golmankhaneh2023solving}).

\begin{example}
Let us consider the fractal differential equation on a fractal set \( F \) given by:
\begin{equation}\label{9oo}
D^{\alpha}_F y(t) = \left(S^{\alpha}_F(t) + y(t)\right)^2.
\end{equation}
It is easy to observe that here the function \( g \) is defined by:
\[
g(t) = \left(S^{\alpha}_F(t) + y(t)\right)^2.
\]
Now, by setting \( z(t) = S^{\alpha}_F(t) + y(t) \), and applying the \( F^{\alpha} \)-derivative to both sides, we obtain:
\[
D^{\alpha}_F z(t) = \chi_F(t) + D^{\alpha}_F y(t), 
\]
Therefore,
\[
D^{\alpha}_F z(t) = \chi_F(t) + \left(S^{\alpha}_F(t) + y(t)\right)^2 = \chi_F(t) + z^2(t).
\]
That is,
\begin{equation}\label{87by}
D^{\alpha}_F z(t) = \chi_F(t) + z^2(t).
\end{equation}

So, solving it by the separation of variables, we obtain:
\[
\arctan z (t) = \int \frac{1}{\chi_F(t) + z^2} \, d^{\alpha}_F z = \int \chi_F(t)d^{\alpha}_F t = S^{\alpha}_F(t) + S^{\alpha}_F(c),
\]
where \( S^{\alpha}_F(c) \) is an arbitrary constant. Thus, solving for \( y (t)\), we obtain the exact solution:
\begin{equation}\label{33322}
y(t) = \tan\left(S^{\alpha}_F(t) + S^{\alpha}_F(c)\right) - S^{\alpha}_F(t).
\end{equation}

Note that in Figure~\ref{fig:tan_minus_SF}, we depict the graphical representation of the solution of Eq.~\eqref{9oo}.
\begin{figure}[H]
    \centering
    \includegraphics[width=0.8\textwidth]{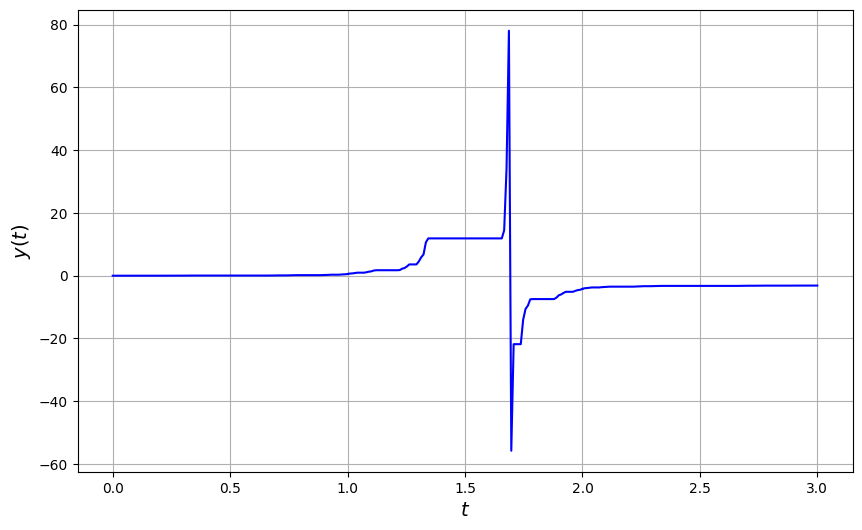}
    \caption{Plot of Eq.~\eqref{33322} over the ternary Cantor set. The solution visualizes the interplay between the fractal staircase function \( S^{\alpha}_F(t) \) and the tangent function.}
    \label{fig:tan_minus_SF}
\end{figure}
\end{example}

\section{Solving Riccati-Type Fractal Differential Equations \label{S-4}}

In this section, we examine, through the proof of several propositions, methods for solving the following class of nonlinear fractal differential equations:

\begin{equation}\label{R333}
D^{\alpha}_F y (t) = a(t)y(t) + b(t)y^2(t) + c(t), \quad \forall t \in F\cap [a,b],
\end{equation}

where  \( a(t) \), \( b(t) \), and \( c(t) \) are \( F \)-continuous functions.
The Eq ~\eqref{R333} is called a Riccati-type fractal differential equation (RFDE) due to its similarity to the classical Riccati equation.

The general solution of RFDE is complex.
Several specific solution techniques are described below.

First of all let us observe that if \( c(t) \equiv 0 \) for all \( t \in F \cap [a,b]\), then Eq.~\eqref{R333} reduces to the fractal Bernoulli differential equation
\[
D^{\alpha}_F y (t) = a(t)y(t) + b(t)y^2(t),
\]
which has already been studied in \cite{golmankhaneh2023solving}. Furthermore, if \( b(t) \equiv 0 \) for all \( t \in F\cap [a,b] \), then Eq.~\eqref{R333} becomes the linear fractal differential equation
\[
D^{\alpha}_F y (t) = a(t)y(t) + c(t), \quad t \in F\cap [a,b],
\]
which has also been investigated in \cite{golmankhaneh2023solving}.

For this reason, from now on we will consider Eq.~\eqref{R333} with both coefficients \( c(t)\) and \(b(t)\) different from zero.

Let us start by describing some solution techniques of RFDE that are based on the knowledge of a particular solution.

\begin{proposition}\label{p1}
Each solution of equation the RFDE (Eq. ~\eqref{R333}) has  the following form
\begin{equation}
y(t) =\,u(t) \,+\,v(t),
\end{equation}
where \( u(t) \) is a particular solution of Eq.~\eqref{R333} while \( v(t) \) is an exact solution of the following Bernoulli-type fractal differential equation:
\begin{equation}\label{B12}
D^{\alpha}_F v(t) = \left(a(t) + 2b(t)u(t)\right)v(t) + b(t)v^2(t).
\end{equation}

\end{proposition}

\begin{proof}
Let \( u(t) \) be a particular solution of Eq.~\eqref{R333}, and let \( v(t) \) be an exact solution of Eq.~\eqref{B12}. We aim to show that
\begin{equation}\label{uv}
y(t) = u(t) + v(t)
\end{equation}
is a general solution of Eq.~\eqref{R333}. Therefore, applying Proposition \ref{pp1} to both members of Eq.~\eqref{uv}  and requiring that  the function $u(t)+v(t)$  satisfies  Eq.~\eqref{R333}, we get:
\begin{align}\label{222}
D^{\alpha}_F y(t) &= D^{\alpha}_F u(t) + D^{\alpha}_F v(t) \\
&= a(t)(u(t) + v(t)) + b(t)(u(t) + v(t))^2 + c(t) \nonumber \\
&= \left[a(t)u(t) + b(t)u^2(t) + c(t)\right] + a(t)v(t) + 2b(t)u(t)v(t) + b(t)v^2(t). \nonumber
\end{align}
Since \( u(t) \) satisfies Eq.~\eqref{R333} and \( v(t) \) satisfies Eq.~\eqref{B12}, the equation holds true. Hence, \( y(t) = u(t) + v(t) \) is indeed an exact solution of Eq.~\eqref{R333}.
\end{proof}

Unfortunately, there is no general algorithm for finding the particular solution \( u(t) \), as it depends on the specific forms of the functions \( a(t) \), \( b(t) \), and \( c(t) \).

In the following, we present an example to illustrate the method.

\begin{example} \label{999moz}
Consider the following RFDE on a fractal set \( F \subset [0,1] \):
\begin{equation}\label{r1}
D^{\alpha}_F y(t) + 2S^{\alpha}_F(t)y(t) = \chi_{F}(t) + \left(S^{\alpha}_F(t)\right)^2 + y^2(t).
\end{equation}
Here the functions \( a(t) \), \( b(t) \), and \( c(t) \) are respectively: \( a(t) \,=\,2S^{\alpha}_F(t) \), \( b(t)=1 \) and \( c(t)= \chi_{F}(t) +\left(S^{\alpha}_F(t)\right)^2\). Moreover,  \( u(t) = S^{\alpha}_F(t) \) is a particular solution of Eq.~\eqref{r1}. Indeed, since \( D^{\alpha}_F S^{\alpha}_F(t) = \chi_{F}(t)\); see Remark \ref{*} and \cite{parvate2009calculus, Alireza-book}, it is trivial to verifies that
\begin{equation}
D^{\alpha}_F u(t) \,=\,- 2S^{\alpha}_F(t)u(t) + u^2(t)\,+\,\chi_{F}(t) + \left(S^{\alpha}_F(t)\right)^2.
\end{equation}
To obtain the general solution of Eq.~\eqref{r1}, we apply the proposition \ref{p1}, which requires solving the following  Bernoulli-type differential equation:
\begin{equation}\label{b1}
D^{\alpha}_F v(t) = v^2(t).
\end{equation}
Finally, following methods in \cite{golmankhaneh2023solving}, we get:
\[
- \frac{1}{v(t)} = S^{\alpha}_F(t) + S^{\alpha}_F(c),
\]
where $S^{\alpha}_F(c)$ is a constant.
Hence:
\[
v(t) = -\frac{1}{S^{\alpha}_F(t) + S^{\alpha}_F(c)}.
\]
Therefore, the general solution of Eq.~\eqref{r1} is given by:
\begin{equation}\label{77}
y(t) = S^{\alpha}_F(t) - \frac{1}{S^{\alpha}_F(t) + S^{\alpha}_F(c)}.
\end{equation}

\begin{figure}[H]
    \centering
    \includegraphics[width=0.85\textwidth]{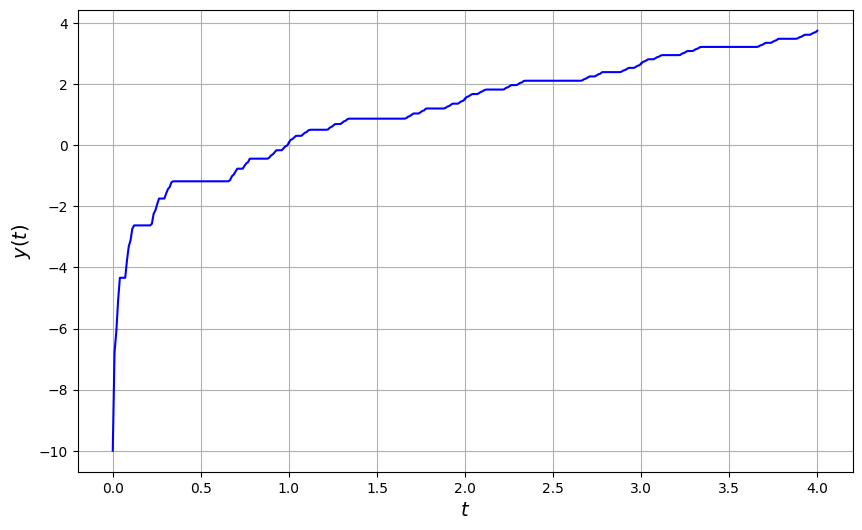}
    \caption{Plot of Eq.~\eqref{77} for a shifted sequence of time series, where \( S^{\alpha}_F(t) \) is the integral staircase function corresponding to the middle-third Cantor set with fractal dimension \( \alpha = \frac{\log 2}{\log 3} \).}
    \label{fig:y_vs_SF_alpha}
\end{figure}
\end{example}

\begin{proposition}\label{p2}
Each solution of equation the RFDE (Eq. ~\eqref{R333}) has  the following form
\begin{equation}
y(t) =\,u(t) \,+\,\frac{1}{v(t)}, \quad \forall t \in F,
\end{equation}
where \( u(t) \) is a particular solution of Eq.~\eqref{R333} while \( v(t) \) is an exact solution of the following linear fractal differential equation:
\begin{equation}\label{B14}
D^{\alpha}_F v(t) = -\left(a(t) + 2b(t)u(t)\right)v(t) -b(t).
\end{equation}

\end{proposition}

\bigskip

The proof of this proposition is left to the reader due to its similarity to the proof of Proposition \ref{p1}.
\begin{remark}\label{rr}
The method described in Proposition \ref{p2} could be more effective than the one in Proposition \ref{p1}, as the function $v(t)$ satisfies a linear fractal differential equation instead of a Bernoulli-type equation. We now demonstrate how Proposition \ref{p2} allows us to solve the RFDE proposed in Example \ref{999moz} more easily and quickly. Indeed, the linear differential equation associated with the RFDE  is: 
$$D^{\alpha}_F v(t)=-1.$$ Therefore the exact solution of Eq. ~\eqref{r1} is \begin{equation}
y(t) = S^{\alpha}_F(t) + \frac{1}{S^{\alpha}_F(c)-S^{\alpha}_F(t)}.
\end{equation}
where $S^{\alpha}_F(c)$ is a constant.

\end{remark}

\medskip
To better describe the method proposed by Proposition \ref{p2}, let us examine the following example:

\begin{example}

Let \begin{equation}\label{r10}
D^{\alpha}_F y(t) =\frac{1}{ S^{\alpha}_F(t)}y(t) + y^2(t) -4 (S^{\alpha}_F(t))^2
\end{equation}
 be a RFDE defined on a fractal subset of the real line $F\subset [0,1].$

 It is trivial to observe that $u(t)=2S^{\alpha}_F(t)$ is a particular solution of Eq.\eqref{r10}.

 Let us show that an exact solution of the proposed RFDE is of the form:
 \begin{equation}\label{1}
y(t) =\,u(t) \,+\,\frac{1}{v(t)}=2S^{\alpha}_F(t)\,+\,\frac{1}{v(t)}, \quad \forall t \in F,
\end{equation}

where the function \( v(t) \) is to be determinate.

Let us $F^{\alpha}$-derive both members of the Eq. \eqref{1}:
$$D^{\alpha}_F y(t)=2\, \chi_{F}(t) -\frac{D^{\alpha}_F v(t)}{v^2(t)}$$ and impose that Eq. \eqref{1} verifies  Eq. \eqref{r10}, so we obtain:

 \begin{equation}\label{111}
2\, \chi_{F}(t) -\frac{D^{\alpha}_F v(t)}{v^2(t)}=\frac{1}{ S^{\alpha}_F(t)}\left (2S^{\alpha}_F(t)\,+\,\frac{1}{v(t)}\right )+ \left (2S^{\alpha}_F(t)\,+\,\frac{1}{v(t)}\right )^2-4(S^{\alpha}_F(t))^2.
\end{equation}

Now, solving Eq.\eqref{111} with respect to $v(t)$, we get:

\begin{equation}\label{000}
D^{\alpha}_F v(t)=-\left(4S^{\alpha}_F(t)+\frac{1}{ S^{\alpha}_F(t)}\right )v(t)-1.
\end{equation}

Finally, following methods in \cite{golmankhaneh2023solving} we have that the exact solution of Eq. \eqref{000} is
\begin{equation}
    v(t)=\frac{e^{(-S^{\alpha}_F(t))^2}}{S^{\alpha}_F(t)}\left ( -\frac{e^{(S^{\alpha}_F(t))^2}}{S^{\alpha}_F(t)}+S^{\alpha}_F(c)\right ),
\end{equation}

where $S^{\alpha}_F(c)$ is a constant.

Therefore the exact solution of Eq. \eqref{r10} is:

\begin{equation}\label{11}
y(t) =\,2S^{\alpha}_F(t)\,+\,\frac{1}{\frac{e^{(-S^{\alpha}_F(t))^2}}{S^{\alpha}_F(t)}\left ( -\frac{e^{(S^{\alpha}_F(t))^2}}{S^{\alpha}_F(t)}+S^{\alpha}_F(c)\right )}, \quad \forall t \in F.
\end{equation}

 \begin{figure}
  \centering
  \includegraphics[width=\textwidth]{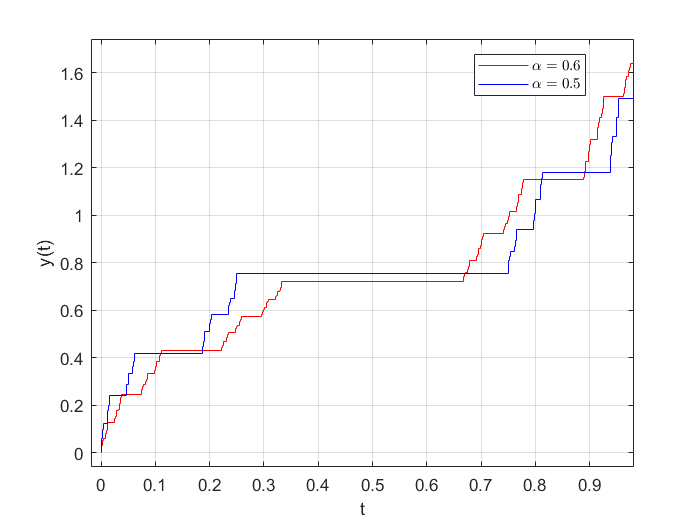}
  \caption{Plot of Eq.\eqref{11} for different fractal subsets of the real line.}\label{err}
\end{figure}
In Figure \ref{err}, we show the effect of the support of the function on the solution with dimensions  $\alpha=0.5$ and $\alpha=0.6$.
\end{example}

\medskip

\begin{proposition}\label{p3}
Let \( b : F\cap [a,b] \to \mathbb{R} \) be a positive function such that \( b \in C^1(F\cap [a,b]) \).
Let \( z(t) \) be a solution of the following second-order fractal differential equation:
\begin{equation}\label{eee2}
D^{2\alpha}_F z(t) = \left( \frac{D^{\alpha}_F b(t)}{b(t)} + a(t) \right) D^{\alpha}_F z(t) - b(t)c(t)z(t), \quad \forall t \in F\cap [a,b],
\end{equation}

Then the  RFDE (Eq. ~\eqref{R333}) has an exact solution of the form
\begin{equation}\label{xxx}
y(t) =\, -\,\frac{D^{\alpha}_F z(t)}{b(t)z(t)}, \qquad \forall t \in F\cap [a,b].
\end{equation}

\end{proposition}

\begin{proof}
Let us suppose that \( z(t) \) is a solution of Eq ~\eqref{eee2} and let us show that
\begin{equation}
y(t) =\, -\,\frac{D^{\alpha}_F z(t)}{b(t)z(t)},
\end{equation}
is a solution of Eq. ~\eqref{R333}.

To do this, let us take the \( F^{\alpha} \)-derivative of both sides of the previous equation:
\begin{align} \label{xxxx}
D^{\alpha}_F y(t) &= D^{\alpha}_F \left( -\frac{D^{\alpha}_F z(t)}{b(t)z(t)} \right) \\
&= \frac{-b(t)z(t)D^{2\alpha}_F z(t) + z(t)D^{\alpha}_F z(t)  D^{\alpha}_F b(t) + b(t) (D^{\alpha}_F z(t))^2}{b^2(t)z^2(t)}\nonumber
\end{align}

and impose that Eq ~\eqref{xxxx} satisfies   Eq. ~\eqref{R333}.

Therefore, we have:
\begin{align}
&\frac{-b(t)z(t)D^{2\alpha}_F z(t) + z(t)D^{\alpha}_F z(t)  D^{\alpha}_F b(t)  + b(t) (D^{\alpha}_F z(t))^2}{b^2(t)z^2(t)} \\&= -\,\frac{a(t)D^{\alpha}_F z}{b(t)z(t)}+ b(t) \left( \frac{D^{\alpha}_F z(t)}{b(t)z(t)} \right)^2 + c(t).
\end{align}

Now, multiplying both sides by \( b(t)z(t) \) and simplifying the equation, we obtain:
\begin{equation}
- D^{2\alpha}_F z(t) + \frac{D^{\alpha}_F b(t)}{b(t)} D^{\alpha}_F z(t) = -\,a(t) D^{\alpha}_F z(t) + b(t)c(t)z(t).
\end{equation}

Finally, rearranging terms gives the required second-order fractal differential equation:
\begin{equation}
D^{2\alpha}_F z(t) = \left( \frac{D^{\alpha}_F b(t)}{b(t)} + a(t) \right) D^{\alpha}_F z(t) - b(t)c(t)z(t), \quad \forall t \in F\cap [a,b].
\end{equation}

\end{proof}

\begin{example}
Let $C$ be the classical Cantor set and let
\begin{equation}\label{e1}
D^{\alpha}_C y(t) = -\frac{3}{S^{\alpha}_C(t)} y(t) + (S^{\alpha}_C(t))^3 y^2 + \frac{1}{(S^{\alpha}_C(t))^3},
\end{equation}

be the  RFDE defined on each point of   \( C \). Here  $\alpha=\frac{\log 3 }{\log 2}$, however, to simplify the notation, we will continue to use $\alpha$ instead of its value. 

It is trivial to notice that \( b(t)= \left (S^{\alpha}_C(t)\right )^3 \in C^1(C) \) on each point of  \( C \),
 therefore, to solve Eq.~\eqref{e1}, we can apply Proposition \ref{p3}.

We know that a solution to Eq.~\eqref{e1} has the following form:
\begin{equation}\label{e2}
y(t) = -\frac{D^{\alpha}_C z(t)}{(S^{\alpha}_C(t))^3 z(t)},
\end{equation}
where $z(t)$ is a solution of the Eq ~\eqref{eee2}.

Now, by taking the \( C^{\alpha}\)-derivative of both sides of Eq.~\eqref{e2}, we obtain:
\begin{equation}\label{e3}
D^{\alpha}_C y(t) = \frac{- D^{2\alpha}_C z(t) \left( (S^{\alpha}_C(t))^3 z(t) \right) + D^{\alpha}_C z(t) \left( 3z(t)(S^{\alpha}_C(t))^2 + (S^{\alpha}_C(t)^3 D^{\alpha}_C z(t) \right)}{(S^{\alpha}_C(t))^6 z(t)^2}.
\end{equation}

Imposing that Eq.~\eqref{e2} satisfies Eq.~\eqref{e1}, we have:
\begin{equation}\label{e4}
D^{\alpha}_C y(t) = \frac{3 D^{\alpha}_C z(t)}{(S^{\alpha}_C(t))^4 z(t)} + (S^{\alpha}_C(t))^3 \frac{(D^{\alpha}_C z(t))^2}{(S^{\alpha}_C(t))^6 z(t)^2} + \frac{1}{(S^{\alpha}_C(t))^3}.
\end{equation}

Finally, equating Eq.~\eqref{e3} with Eq.~\eqref{e4} and simplifying appropriately, we obtain the following second-order homogeneous fractal differential equation with constant coefficients:
\begin{equation}\label{e5}
D^{2\alpha}_C z(t) + z(t) = 0.
\end{equation}

Following the methods examined in \cite{khalili2024fractal}, we find that an exact solution of Eq.~\eqref{e5} is:
\begin{equation}\label{e6}
z(t) = S_C^{\alpha}(c_1)\cos(S_{F}^{\alpha}(t)) + S_C^{\alpha}(c_2) \sin(S_{F}^{\alpha}(t)),
\end{equation}
where \( S_C^{\alpha}(c_1) \) and \( S_C^{\alpha}(c_2) \) are two arbitrary constants.

Thus, by Eq.~\eqref{e6}, to find the general solution to Eq.~\eqref{e1}, it is sufficient to take the \( C^{\alpha} \)-derivative of Eq.~\eqref{e6}:
\begin{equation}\label{e7}
D^{\alpha}_C z(t) = -S_C^{\alpha}(c_1) \sin(S_{F}^{\alpha}(t)) + S_C^{\alpha}(c_1)\cos(S_{F}^{\alpha}(t)).
\end{equation}

Therefore, indicating by \( S_C^{\alpha}(c) = S_C^{\alpha}(c_2) / S_C^{\alpha}(c_1) \)  a constant, an exact general solution of Eq.~\eqref{e1} is:
\begin{equation}\label{e8}
y(t) = \frac{\sin(S_{F}^{\alpha}(t)) -c \cos(S_{F}^{\alpha}(t))}{(S^{\alpha}_C(t))^3 (\cos(S_{F}^{\alpha}(t)) + S_C^{\alpha}(c) \sin(S_{F}^{\alpha}(t)))}.
\end{equation}

Note that Figure \ref{1fig:fractal_y_function} shows the graph of Eq.~\eqref{e8}  with the constant \( S_C^{\alpha}(c) = 1 \).
\begin{figure}[H]
    \centering
    \includegraphics[width=0.8\textwidth]{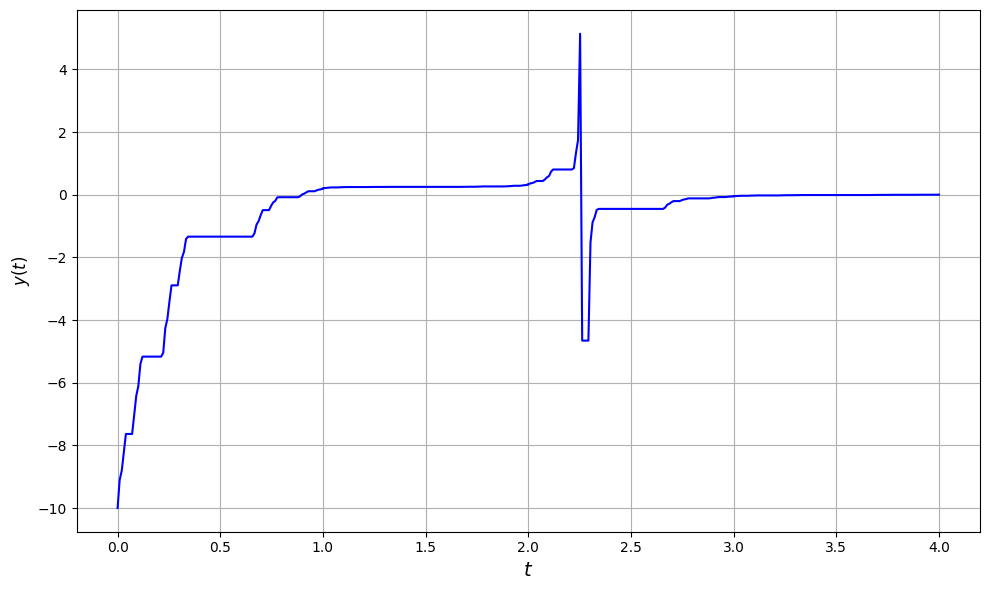}
    \caption{Plot of Eq.~\eqref{e8}, where \( S^{\alpha}_F(t) \) is the integral staircase function based on the middle-third Cantor set.}
    \label{1fig:fractal_y_function}
\end{figure}
\end{example}

\section{Fractal Riccati Formulation of the Schr\"{o}dinger Equation \label{S-5}}

In this section we show how the Riccati-type fractal diffeerntial equation (RDFE) intervenes in the formulation of the Schr\"{o}dinger equation on fractal domains.
The time-independent $\alpha$-dimensional Schr\"{o}dinger fractal equation \cite{golmankhaneh2024fractalEE}   on an $\alpha$ perfect fractal set $F$ is given by
\begin{equation}\label{SchrodingerFractal}
-\frac{\hbar^{2}}{2m}D_{F}^{2\alpha}\psi(t) + V(t)\psi(t) = E\psi(t), \quad \forall t \in F,
\end{equation}
where $V(t)$ is the potential, $\psi(t)$ the wavefunction, and $E$ the energy eigenvalue.

An efficient method for its study is the \emph{factorization approach} \cite{bagchi2000supersymmetry}, in which the Hamiltonian is written as a product of first $\alpha$-order operators:
\begin{equation}
A = D_{F}^{\alpha} + W(t), \qquad A^{\dagger} = -D_{F}^{\alpha} + W(t),
\end{equation}
where $D_{F}^{\alpha}$ is the $F^{\alpha}$-derivative operator and $W(t)$ is the \emph{superpotential} \cite{gangopadhyaya2017supersymmetric}. The Hamiltonian becomes
\begin{equation}
H = A^{\dagger}A = -D_{F}^{2\alpha} + W(t)^{2} - D_{F}^{\alpha}W(t).
\end{equation}

Thus the potential is related to $W(t)$ through the RFDE:
\begin{equation}\label{FractalRiccati}
V(t) = W(t)^{2} - D_{F}^{\alpha} W(t), \quad \forall t \in F.
\end{equation}

The superpotential itself is connected with the ground-state wavefunction $\psi_{0}(t)$ via
\begin{equation}
W(t) = -D_{F}^{\alpha}\ln \psi_{0}(t).
\end{equation}
This provides a direct link between the RFDE and the quantum system under consideration.

\begin{example}
Let $F\subset \R$ be an $\alpha$-perfect fractal set.
Let us consider the fractal harmonic oscillator potential for the Harmonic Oscillator
\begin{equation}
V(x) = \tfrac{1}{2}m\omega^{2}S_{F}^{\alpha}(x)^{2}, \quad \forall x \in F.
\end{equation}
Note that here $S_{F}^{\alpha}(x)$ denotes the $\alpha$-dimensional fractal coordinate \cite{golmankhaneh2024fractalEE}.

The normalized ground-state wavefunction \cite{merzbacher1998quantum8}, is
\begin{equation}
\psi_{0}(x) \propto \exp\!\left(-\tfrac{m\omega}{2\hbar} S_{F}^{\alpha}(x)^{2}\right).
\end{equation}
From the logarithmic derivative, the superpotential is
\begin{equation}
W(x) = -D_{F}^{\alpha}\ln \psi_{0}(x)
     = \sqrt{\tfrac{m\omega}{2\hbar}}\,S_{F}^{\alpha}(x).
\end{equation}

Substitution into \eqref{FractalRiccati} confirms that this $W(x)$ satisfies the RFDE, showing that the fractal harmonic oscillator is exactly solvable by the factorization method.
\end{example}

\medskip

\begin{example}

Let $F\subset \R$ be an $\alpha$-perfect fractal set.
Let us consider the fractal Coulomb potential for the Hydrogen atom \cite{merzbacher1998quantum8}:
\begin{equation}
V(r) = -\frac{e^{2}}{4\pi \epsilon_{0}} \frac{1}{S_{F}^{\alpha}(r)},  \quad \forall r \in F , \quad r>0.
\end{equation}

In atomic units ($\hbar=m=e^{2}/4\pi\epsilon_{0}=1$), the effective radial equation for $u(r)=S_{F}^{\alpha}(r)R_{n\ell}(r)$ reads
\begin{equation}
-\tfrac{1}{2}D_{F}^{2\alpha} u(r)
+ \left[ \frac{\ell(\ell+1)}{2S_{F}^{\alpha}(r)^{2}} - \frac{1}{S_{F}^{\alpha}(r)} \right]u(r)
= Eu(r).
\end{equation}

The RFDE for the superpotential $W_{\ell}(r)$ is
\begin{equation}\label{CoulombRiccati}
W_{\ell}(r)^{2} - D_{F}^{\alpha} W_{\ell}(r)
= \frac{\ell(\ell+1)}{S_{F}^{\alpha}(r)^{2}} - \frac{2}{S_{F}^{\alpha}(r)} - E_{0}, \quad \forall r \in F.
\end{equation}

For the ground state ($n=1,\ell=0$) \cite{merzbacher1998quantum8}, the radial solution is
\begin{equation}
u_{10}(r) \propto \exp(-S_{F}^{\alpha}(r)),
\end{equation}
giving
\begin{equation}
W_{0}(r) = -D_{F}^{\alpha}\ln u_{10}(r).
\end{equation}

For general $\ell$, one finds
\begin{equation}
W_{\ell}(r) = \frac{\ell+1}{S_{F}^{\alpha}(r)} - \frac{1}{\ell+1},
\end{equation}
which solves \eqref{CoulombRiccati}. The partner potentials are then
\begin{align}
V_{-}(r) &= \frac{\ell(\ell+1)}{S_{F}^{\alpha}(r)^{2}} - \frac{2}{S_{F}^{\alpha}(r)} - \frac{1}{(\ell+1)^{2}}, \\
V_{+}(r) &= \frac{(\ell+1)(\ell+2)}{S_{F}^{\alpha}(r)^{2}} - \frac{2}{S_{F}^{\alpha}(r)} - \frac{1}{(\ell+1)^{2}}.
\end{align}

This demonstrates the \emph{shape invariance} of the Coulomb potential under supersymmetric factorization, with $V_{+}(r)$ corresponding to the effective potential of angular momentum $\ell+1$. The Riccati-type fractal formulation thus provides a powerful algebraic route to the hydrogen atom spectrum.
\end{example}

\section{Conclusion \label{S-6}}

In this paper, the $F^\alpha$-calculus  was studied and utilized to obtain the exact solutions of some non-linear fractal differential equations.
The examples discussed not only demonstrate the effectiveness of the proposed approach but also open new research directions for the application of $F^{\alpha}$-calculus in non-linear differential equations, with potential applications across various scientific and engineering fields, extending the power of fractal methods to more complex non-linear systems.\\
\textbf{Declaration of Competing Interest:}\\
The authors declare that they have no known competing financial interests or personal relationships that could have appeared to influence the work reported in this paper.\\
\textbf{Data Availability Statement:}\\ No data were generated or analyzed during the current study.\\

\bibliographystyle{elsarticle-num}
\bibliography{fadivnonhomo2}

\begin{thebibliography}{10}
\expandafter\ifx\csname url\endcsname\relax
  \def\url#1{\texttt{#1}}\fi
\expandafter\ifx\csname urlprefix\endcsname\relax\def\urlprefix{URL }\fi
\expandafter\ifx\csname href\endcsname\relax
  \def\href#1#2{#2} \def\path#1{#1}\fi

\bibitem{Coddington}
E.~A. Coddington, N.~Levinson, Theory of Ordinary Differential Equations, McGraw-Hill, New York, 1955.

\bibitem{Verulst}
F.~Verhulst, Nonlinear Differential Equations and Dynamical Systems, Springer-Verlag, 1996.

\bibitem{edwards2000differential}
C.~H. Edwards, D.~E. Penney, Differential equations and boundary value problems: computing and modeling, Pearson Educaci{\'o}n, 2000.

\bibitem{bellon2008riccati}
J.~Bellon, Riccati equations in optimal control theory, Ph.D. thesis, Georgia State University (2008).
\newblock \href {https://doi.org/10.57709/1059702} {\path{doi:10.57709/1059702}}.

\bibitem{Martin}
C.~Martin, Finite escape time for {R}iccati differntial equations, System \& Control Letters 1~(7) (1981) 127--131.

\bibitem{Pavon}
P.~E. Crouch, M.~Pavon, On the existence of solutions of the {R}iccati differential equation, Systems \& Control Letters 9~(3) (1987) 203--206.

\bibitem{Mandelbro}
B.~B. Mandelbrot, The Fractal Geometry of Nature, W.H. Freeman, New York, 1982.

\bibitem{Qaswet}
P.~R. Massopust, Fractal Functions, Fractal Surfaces, and Wavelets, Academic Press, Massachusetts, 2017.

\bibitem{falconer1999techniques}
K.~Falconer, Fractal Geometry: Mathematical Foundations and Applications, John Wiley \& Sons, New York, 2004.

\bibitem{feder2013fractals}
J.~Feder, Fractals, Springer Science \& Business Media, New York, 2013.

\bibitem{kigami1989harmonic}
J.~Kigami, A harmonic calculus on the {S}ierpinski spaces, Japan J. Appl. Math. 6~(2) (1989) 259--290.

\bibitem{kigami2001analysis}
J.~Kigami, Analysis on Fractals, no. 143, Cambridge University Press, Cambridge, 2001.

\bibitem{jorgensen2006analysis}
P.~E. Jorgensen, Analysis and Probability: Wavelets, Signals, Fractals, Vol. 234, Springer Science \& Business Media, New York, 2006.

\bibitem{jiang1998some}
H.~Jiang, W.~Su, Some fundamental results of calculus on fractal sets, Commun. Nonlinear Sci. Numer. Simul. 3~(1) (1998) 22--26.

\bibitem{bongiorno2015integral}
D.~Bongiorno, G.~Corrao, An integral on a complete metric measure space, Real Anal. Exch. 40~(1) (2015) 157--178.

\bibitem{Barlowqq1}
M.~T. Barlow, E.~A. Perkins, Brownian {M}otion on the {S}ierpinski {G}asket, Probab. Th. Rel. Fields 79~(4) (1988) 543--623.

\bibitem{valarmathi2023variable}
R.~Valarmathi, A.~Gowrisankar, On the variable order fractional calculus of fractal interpolation functions, Fract. Calc. Appl. Anal. 26~(3) (2023) 1273--1293.

\bibitem{parvate2009calculus}
A.~Parvate, A.~D. Gangal, Calculus on fractal subsets of real line-{I}: Formulation, Fractals 17~(01) (2009) 53--81.

\bibitem{parvate2011calculus}
A.~Parvate, S.~Satin, A.~Gangal, Calculus on fractal curves in $\mathbb{R}^{n}$, Fractals 19~(01) (2011) 15--27.

\bibitem{Alireza-book}
A.~K. Golmankhaneh, Fractal Calculus and its Applications, World Scientific, Singapore, 2022.

\bibitem{golmankhaneh2023solving}
A.~K. Golmankhaneh, D.~Bongiorno, Exact solutions of some fractal differential equations, Appl. Math. Comput. 472 (2024) 128633.

\bibitem{khalili2024fractal}
A.~K. Golmankhaneh, D.~Bongiorno, Fractal differential equations of 2 $\alpha$-order, Axioms 13~(11) (2024) 786.

\bibitem{golmankhaneh2025homogeneous}
A.~K. Golmankhaneh, D.~Bongiorno, On homogeneous system of fractal differential equations, Journal of Mathematical Sciences (2025).
\newblock \href {https://doi.org/10.1007/s10958-025-07658-8} {\path{doi:10.1007/s10958-025-07658-8}}.

\bibitem{khalili2025fractal77}
A.~K. Golmankhaneh, D.~Bongiorno, A.~T. Ramazanova, Fractal calculus: nonhomogeneous linear systems, Journal of Nonlinear, Complex and Data Science (2025).
\newblock \href {https://doi.org/10.1016/j.amc.2024.128633} {\path{doi:10.1016/j.amc.2024.128633}}.

\bibitem{khalili2024fractaldd}
A.~K. Golmankhaneh, K.~Welch, C.~Serpa, P.~E. J{\o}rgensen, Fractal {M}ellin transform and non-local derivatives, Georgian Math. J. 31~(3) (2024) 423--436.

\bibitem{Uc}
M.~Uc, Spectral and algebraic analysis of the fractal {V}olterra operator on {C}k({F}), Chaos, Solitons \& Fractals 200 (2025) 117061.

\bibitem{golmankhaneh2018sub}
A.~K. Golmankhaneh, A.~S. Balankin, Sub-and super-diffusion on {C}antor sets: Beyond the paradox, Phys. Lett. A. 382~(14) (2018) 960--967.

\bibitem{golmankhaneh2024fractalEE}
A.~K. Golmankhaneh, S.~Pellis, M.~Zingales, Fractal schr{\"o}dinger equation: implications for fractal sets, Journal of Physics A: Mathematical and Theoretical 57~(18) (2024) 185201.

\bibitem{bagchi2000supersymmetry}
B.~K. Bagchi, Supersymmetry in quantum and classical mechanics, CRC Press, 2000.

\bibitem{gangopadhyaya2017supersymmetric}
A.~Gangopadhyaya, J.~V. Mallow, C.~Rasinariu, Supersymmetric quantum mechanics: An introduction, World Scientific Publishing Company, 2017.

\bibitem{merzbacher1998quantum8}
E.~Merzbacher, Quantum mechanics, John Wiley \& Sons, 1998.

\end{thebibliography}

\end{document}